\documentclass[12pt]{article}

\usepackage{amsmath}
\usepackage{amssymb}
\usepackage{amsthm}
\usepackage{newtxtext,newtxmath}
\usepackage{url}
\usepackage{longtable}

\theoremstyle{definition}
\newtheorem{theorem}{Theorem}[section]
\newtheorem{lemma}[theorem]{Lemma}
\newtheorem{corollary}[theorem]{Corollary}
\newtheorem{remark}[theorem]{Remark}
\newtheorem{definition}[theorem]{Definition}

\title{An analogue of the Robin inequality of the second type for odd integers}
\date{7 March, 2022}
\author{Yoshihiro Koya,\\
Yokohama City University}

\begin{document}
\maketitle

\begin{abstract}
  In this paper we give a variant of the Robin inequality which states that
$\frac{\sigma(n)}{n} \leq \frac{e^\gamma}{2} \log\log n
  + \frac{0.7398\cdots}{\log\log n}$
  for any odd integer $n \geq 3$.
\end{abstract}

\section{Introduction}

The Riemann zeta function is defined by
\begin{equation*}
  \zeta(s) = \sum^\infty_{n=1} \dfrac{1}{n^s} \qquad (\mathrm{Re}(s) > 1),
\end{equation*}
and it is analytically continued to the whole complex plane.
The Riemann zeta function has two kinds of zeros.
The first are called trivial zeros.
They are zeros at $-2, -4, \ldots, -2n, \ldots$.
On the other hand it is well-known that
the Riemann zeta function has infinite number of zeros
on the line $\mathrm{Re}(s) = \frac{1}{2}$ (\cite{Titch}).
They are called complex zeros.

The Riemann hypothesis states that
all complex zeros of the Riemann zeta function lie on the line
$\mathrm{Re}(s) = \frac{1}{2}$.
Many mathematicians have made attempts to solve the Riemann hypothesis.
And several criteria which are equivalent to the Riemann hypothesis were
obtained by some authors.
One of such a criterion is the Robin inequality.

In \cite{robin}, G.~Robin showed the following theorem~:

\begin{theorem}[Theorem~1 of \cite{robin}]
  The Riemann hypothesis equivalent
  to the statement that
  \begin{equation*}
    \dfrac{\sigma(n)}{n} < e^\gamma \log\log n 
  \end{equation*}
  for any integer $n> 5040$,
  where $\sigma(n)$ stands for the sum of divisors of $n$ and
  $\gamma$ is the Euler-Mascheroni constant.
\end{theorem}

There are some variants of this criterion.
One of them is the Washington-Yang inequality.
That is, L.~C.~Washington and A.~Yang showed
the following theorem~:

\begin{theorem}[Theorem~1 of \cite{WYang}]
  The Riemann hypothesis is equivalent to the statement that
  \begin{equation*}
    \dfrac{\sigma(n)}{n} < \dfrac{e^\gamma}{2}\log\log n
  \end{equation*}
  for all odd integers $n \geq 3^4 \cdot 5^3 \cdot 7^2 \cdot 11 \cdot 13 \cdots 67
  = 18565284664427130919514350125$.
\end{theorem}

On the other hand G.~Robin also proved the following theorem, which
holds independently of the Riemann hypothesis~:

\begin{theorem}[Theorem 2 of \cite{robin}]
  One has that
  \begin{equation*}
        \dfrac{\sigma(n)}{n} < e^\gamma \log\log n + \dfrac{0.6482 \cdots}{\log\log n}
  \end{equation*}
  for all $n \geq 3$ with equality for $n=12$.
\end{theorem}
We shall call the inequality of this type the Robin inequality of the second type in this paper.

In this paper, we prove the following theorem~:

\begin{theorem}[Theorem~\ref{theorem:main}]
  We have that
  \begin{equation*}
    \dfrac{\sigma(n)}{n} \leq \dfrac{e^\gamma}{2} \log\log n
    + \dfrac{0.7398\cdots}{\log\log n}
  \end{equation*}
  for odd integer $n \geq 3$.

  Moreover, if $n=315$, then equality holds.
\end{theorem}
Here the constant $0.7398\cdots$ is defined by the following~:
\begin{equation*}
  0.7398002037224377\cdots = \left( \dfrac{\sigma(315)}{315} - \dfrac{e^\gamma}{2} \log\log 315 \right)
  \times \log\log 315.
\end{equation*}

Note that this inequality can be considered as an analogue of the above Robin's second inequality
in the direction of Washington and Yang.

\section{Some Lemmas}

In this section we prepare some lemmas and a theorem, which are shown 
by T.~Oshiro (\cite{oshiro}).

Let $p_k$ be the $k$-th prime.  Thus we have $p_1 = 2, p_2 = 3, p_3 = 5, \ldots$.
Let $N_k = \prod^{k}_{i = 1} p_i$ be the $k$-th prime factorial and $N_k^\prime = \prod^{k}_{i = 2} p_i$ be the corresponding odd prime factorial.
By its definition we have $N_k = 2 N_k^\prime$.

\begin{lemma}[Lemma~2.5 of \cite{oshiro}] \label{lemma:phi}
  Let $n$ be an odd integer with $N_k^\prime \leq n < N_{k+1}^\prime$.
  Then we have
  \begin{equation*}
    \dfrac{n}{\varphi(n)} \leq \dfrac{N_k^\prime}{\varphi(N_k^\prime)}
  \end{equation*}
\end{lemma}

\begin{lemma}[Lemma~2.2 of \cite{oshiro}] \label{logprod}
  Let $x \geq 10^4$. Then we have
  \begin{equation*}
    \prod_{i=2}^{k}\left(\dfrac{p_i}{p_i-1}\right)
    \leq \dfrac{e^\gamma}{2} \log x \left(1 + \dfrac{0.2}{\log^2 x}\right).
  \end{equation*}
\end{lemma}

\begin{lemma}[Lemma 2.4 of \cite{oshiro}] \label{loglog}
  Let $N^\prime_k$ be the $k$-th odd prime factorial with $p_k \geq 20000$.
  Then we have
  \begin{equation*}
    \log\log N_{k}^\prime > \log p_k - \dfrac{0.216265\cdots}{\log p_k}
  \end{equation*}
\end{lemma}

\begin{proof}
  From \cite[p.206]{robin} we have
  \begin{equation*}
    \log N_k = \theta(p_k) > p_k\left( 1 - \dfrac{1}{8 \log p_k}\right)
  \end{equation*}
  Hence, we have
  \begin{align*}
    \log\log N_k^\prime
    &> \log\left(p_k\left(1 - \dfrac{1}{8 \log p_k}\right) - \log 2\right) \\
    &= \log p_k + \log\left( 1 - \left(\dfrac{1}{8\log p_k} + \dfrac{\log 2}{p_k}\right)\right)
  \end{align*}
  Since $p_k \geq 20000$, we know $p_k \geq 8\log p_k$.
  \begin{equation*}
    \dfrac{1}{8 \log p_k} + \dfrac{\log 2}{p_k}
    \leq \dfrac{1}{8 \log p_k} + \dfrac{\log 2}{8 \log p_k}.
  \end{equation*}
  Therefore 
  \begin{equation*}
    \log\left( 1 - \left(\dfrac{1}{8\log p_k} + \dfrac{\log 2}{p_k}\right)\right)
    > \log\left( 1 - \dfrac{1 + \log 2}{8\log p_k}\right).
  \end{equation*}
  From this, we obtain
  \begin{align*}
    \log\log N_k^\prime
    &> \log p_k
    + \log\left( 1 - \dfrac{1 + \log 2}{8\log p_k} \right) \\
    &= \log p_k
    - \left\{ \dfrac{1 + \log 2}{8\log p_k}
      + \dfrac{1}{2} \left(\dfrac{1 + \log 2}{8\log p_k}\right)^2
      + \dfrac{1}{3} \left(\dfrac{1 + \log 2}{8\log p_k}\right)^3 + \cdots \right\}\\ 
    &> \log p_k - \left\{ \dfrac{1 + \log 2}{8\log p_k}
      + \left(\dfrac{1 + \log 2}{8\log p_k}\right)^2
      + \left(\dfrac{1 + \log 2}{8\log p_k}\right)^3 + \cdots \right\}\\
    &>  \log p_k -  \dfrac{1 + \log 2}{8\log p_k}
    \cfrac{1}{1 - \cfrac{1 + \log 2}{8\log 20000}} \\
    &> \log p_k - \dfrac{0.125 \times ( 1 + \log 2 ) \times 1.02183726\cdots}{\log p_k} \\
    &= \log p_k - \dfrac{0.21626511\cdots}{\log p_k}
  \end{align*}
\end{proof}

The following theorem states that Theorem~\ref{theorem:main} holds for sufficiently large odd $n$.

\begin{theorem}[Theorem~3.1 of \cite{oshiro}] \label{theorem:oshiro}
  Let $N_k^\prime = \prod_{i=2}^{k} p_i$ with $p_k \geq 20000$.
  Then we have
  \begin{equation*}
    \dfrac{\sigma(n)}{n} \leq \dfrac{e^\gamma}{2} \log\log n
    + \dfrac{0.7398\cdots}{\log\log n}
  \end{equation*}
  for odd $n \geq N_k^\prime$.

  Moreover, if $n=315$, then equality holds.
\end{theorem}

\begin{proof}
  From Lemma~\ref{loglog}, we know
  \begin{equation*}
    \log \log N_k^\prime > \log p_k - \dfrac{0.216\cdots}{\log p_k}
  \end{equation*}
  Since the function $t \mapsto \frac{e^\gamma}{2}t + \frac{0.7398\cdots}{t}$ is
  increasing for $t \geq 1$, we have
  \begin{align*}
    \dfrac{e^\gamma}{2}\log\log N_k^\prime + \frac{0.7398\cdots}{\log\log N_k^\prime}
    &> \dfrac{e^\gamma}{2}\left(\log p_k - \dfrac{0.216\cdots}{\log p_k} \right)
    + \dfrac{0.7398\cdots}{\log p_k} \\
    &= \dfrac{e^\gamma}{2} \log p_k - \dfrac{e^\gamma}{2} \dfrac{0.216\cdots}{\log p_k}
    + \dfrac{0.7398\cdots}{\log p_k} \\
    &= \dfrac{e^\gamma}{2} \log p_k + \dfrac{0.5472\cdots}{\log p_k} \\
    &= \dfrac{e^\gamma}{2} \log p_k \left( 1 + \dfrac{0.6144\cdots}{\log^2 p_k} \right)\\
    &> \dfrac{e^\gamma}{2} \log p_k \left(1  + \dfrac{0.2}{\log^2 p_k} \right)\\
    &\geq \prod_{i=2}^k \left( 1 - \dfrac{1}{p_k}\right)^{-1}.
  \end{align*}
  Note that we use Lemma~\ref{loglog} and Lemma~\ref{logprod} in the above.

  Therefore we have
  \begin{equation*}
    \dfrac{e^\gamma}{2}\log\log N_k^\prime + \frac{0.7398\cdots}{\log\log N_k^\prime}
    > \dfrac{N_k^\prime}{\varphi(N_k^\prime)}.
  \end{equation*}
  From this we obtain
  \begin{align*}
    \dfrac{\sigma(n)}{n} < \dfrac{n}{\varphi(n)}
    < \dfrac{N_k^\prime}{\varphi(N_k^\prime)}
    &< \dfrac{e^\gamma}{2}\log\log N_k^\prime + \frac{0.7398\cdots}{\log\log N_k^\prime} \\
    &< \dfrac{e^\gamma}{2}\log\log n + \frac{0.7398\cdots}{\log\log n}.
  \end{align*}
  This completes the proof of the theorem.
\end{proof}

\section{The Main Theorem}

In this section we shall state the main theorem and complete its proof.

\begin{theorem} \label{theorem:main}
  We have
  \begin{equation*}
    \dfrac{\sigma(n)}{n} \leq \dfrac{e^\gamma}{2} \log\log n
    + \dfrac{0.7398\cdots}{\log\log n}
  \end{equation*}
  for odd integer $n \geq 3$.

  Moreover, if $n=315$, then equality holds.
\end{theorem}

In order to prove this, we achieve some numerical computations with prime factorials and odd colossally abundant numbers.

Firstly we observe odd prime factorials.

\begin{lemma}
  We have that
  \begin{equation*}
    \dfrac{\sigma(n)}{n} \leq \dfrac{e^\gamma}{2} \log\log n
    + \dfrac{0.7398\cdots}{\log\log n}
  \end{equation*}
for all odd integer $n \geq N_{54} = 3 \cdot 5 \cdots  251$.
\end{lemma}

\begin{proof}
  Suppose that $N^\prime_{k} \leq n < N^\prime_{k+1}$ for some $k$.
  
  Recall that the function $t \mapsto \frac{e^\gamma}{2}t + \frac{0.7398\cdots}{t}$ is
  increasing for $t \geq 1$.
  Then we have 
  \begin{equation*}
    \dfrac{e^\gamma}{2}\log\log N_k^\prime + \frac{0.7398\cdots}{\log\log N_k^\prime} 
    < \dfrac{e^\gamma}{2}\log\log n + \frac{0.7398\cdots}{\log\log n}.
  \end{equation*}
  On the other hand, from Lemma~\ref{lemma:phi} we have that
  $\dfrac{n}{\varphi(n)} \leq \dfrac{N_k^\prime}{\varphi(N_k^\prime)}$.
  Thus if we have 
  \begin{equation*}
    \dfrac{N_k^\prime}{\varphi(N_k^\prime)}
    \leq
    \dfrac{e^\gamma}{2}\log\log N_k^\prime + \frac{0.7398\cdots}{\log\log N_k^\prime} ,
  \end{equation*}
  we can deduce that the inequality is still valid for $n$.

But, from computation by a computer, we observe that the above inequality is satisfied
for  odd prime factorials greater than
$N_{54}^\prime = 3 \cdot 5 \cdots 251$.
\end{proof}

Before proceeding we give a definition.

\begin{definition}[cf.~\cite{WYang}]
  Let $N$ be an odd integer.
  $N$ is {\em odd colossally abundant} if there exists $\epsilon > 0$
  such that
  \begin{equation*}
    \dfrac{\sigma(n)}{n^{1+\epsilon}} < \dfrac{\sigma(N)}{N^{1+\epsilon}}
  \end{equation*}
  for all odd integers $n \geq 3$ with $n \ne N$.
\end{definition}

\begin{remark}
  In \cite{WYang}, it is pointed out that the inequality
  \begin{equation*}
    \dfrac{\sigma(n)}{n} \leq \dfrac{e^\gamma}{2} \log\log n
  \end{equation*}
  is hold for all odd integer $n \geq 3^4 \cdot 5^3 \cdot 7^2 \cdot 11 \cdots 61$.
  Threfore we already know that our inequality is also valid for such odd numbers.

  We shall, however, give an another approach here.
\end{remark}

\begin{lemma} \label{lemma:partial}
  Let $N$ and $N^\prime$ be consecutive odd colossally abundant numbers.
  
  \begin{enumerate}
  \item Suppose that both $N$ and $N^\prime$ satisfy the inequalities
    \begin{align*}
      \alpha \dfrac{\sigma(N)}{N} &\leq A \log\log N, \\
      \alpha \dfrac{\sigma(N^\prime)}{N^\prime} &\leq A \log\log N^\prime, 
    \end{align*}
    where $A$ and $\alpha$ are fixed positive constants.

    Then any odd integer $n$ with $N \leq n \leq N^\prime$ also satisfies the same inequality
    \begin{equation*}
      \alpha \dfrac{\sigma(n)}{n} \leq A \log\log n.
    \end{equation*}

  \item Suppose that $N^\prime$ satisfies the inequality
    \begin{equation*}
      \beta \dfrac{\sigma(N^\prime)}{N^\prime} \leq \dfrac{B}{\log\log N^\prime},
    \end{equation*}
    where $B$ and $\beta$ are fixed positive constants.

    Then any odd integer $n$ with $n \leq N^\prime$ also satisfies the same inequality
    \begin{equation*}
      \beta \dfrac{\sigma(n)}{n} \leq \dfrac{B}{\log\log n}.
    \end{equation*}
  \item Besides the above assumptions, if $\alpha + \beta \leq 1$, then
    we have 
    \begin{equation*}
      \dfrac{\sigma(n)}{n} \leq A \log\log n + \dfrac{B}{\log\log n}
    \end{equation*}
    for all odd integers $n$ with $N \leq n \leq N^\prime$.
  \end{enumerate}
\end{lemma}

\begin{proof}
  The proof of (1) is almost same as \cite[Proposition~1 of Section~3]{robin}.
  (2) is easy. (3) is also easy consequence of (1) and (2).
\end{proof}

We can verify the following corollary by direct computation.

\begin{corollary} \label{cor:final}
  \begin{enumerate}
  \item Let $\alpha = \frac{19.5}{20}$ and $\beta = \frac{0.5}{20}$
    in Lemma~\ref{lemma:partial}. Then we have that the inequality
    holds for
    $ 3^3 \cdot 5^2 \cdot 7 \cdot 11 \cdots 31 
    \leq n \leq
    3^6\cdot 5^4\cdot 7^3\cdot 11^2\cdot 13^2\cdot 17^2\cdot 19^2\cdot 23^1\cdot 29 \cdots 251$.

  \item Let $\alpha = \frac{19}{20}$ and $\beta = \frac{1}{20}$
    in Lemma~\ref{lemma:partial}. Then we have that the inequality
    holds for
    $ 3^3 \cdot 5^2 \cdot 7 \cdots 17
    \leq n \leq
    3^4 \cdot 5^3 \cdot 7^2 \cdot 11 \cdots 59$.

  \item Let $\alpha = \frac{19}{21}$ and $\beta = \frac{2}{21}$
    in Lemma~\ref{lemma:partial}. Then we have that the inequality
    holds for
    $ 3^2 \cdot 5 \cdot 7 \cdots 13
    \leq n \leq
    3^3 \cdot 5^2 \cdot 7 \cdots 17$.
  \end{enumerate}
\end{corollary}

For the rest case we can verify by using computer.

\begin{lemma}
  For $ 3 \leq n \leq 45045 = 3^2 \cdot 5 \cdot 7 \cdots 13$
  the inequality
    \begin{equation*}
      \dfrac{\sigma(n)}{n} \leq \dfrac{e^\gamma}{2} \log\log n + \dfrac{0.7398\cdots}{\log\log n}
    \end{equation*}
  is still valid.
\end{lemma}

\bigskip
{\small
  \begin{flushleft}
    Yoshihiro Koya\\
    \medskip
    Institute of Natural Science,\\
    Yokohama City University,\\
    22-2 Seto, Kanazawa-ku,\\
    Yokohama 236-0027, JAPAN\\
    \medskip
    \url{koya@yokohama-cu.ac.jp}
  \end{flushleft}}
\end{document}